\documentclass[11pt]{article}

\usepackage{amsmath}
\usepackage{amssymb}
\usepackage{amsfonts}
\usepackage{mathrsfs}
\usepackage[usenames]{color}
\usepackage[numbers]{natbib}
\usepackage{enumitem} 

\usepackage{graphicx}
\usepackage{epsfig}
\usepackage{amsthm}
\usepackage{mathtools}

\usepackage{latexsym}
\newsavebox{\toy}
\savebox{\toy}{\framebox[0.65em]{\rule{0cm}{1ex}}}
\newcommand{\QED}{\usebox{\toy}\end{demo}}

\numberwithin{equation}{section}

\newtheorem{theorem}{Theorem}[section]
\newtheorem{lemma}[theorem]{Lemma}
\newtheorem{proposition}[theorem]{Proposition}
\newtheorem{cor}[theorem]{Corollary}

\def\qed{\hfill\rule{.2cm}{.2cm}\par\medskip\par\relax}

\newcommand{\bd}{\begin{displaymath}}
\newcommand{\ed}{\end{displaymath}}
\newcommand{\R}{{\mathbb{R}}}
\newcommand{{\rd}}{\R^d}

\newcommand{\IP}{{\mathbb P}}
\newcommand{\E}{\mathbb E}

\newcommand{\8}{\infty}

\newcommand{\un}{1 \hspace{-0.6ex}{\rm I}}
\newcommand{\eu }{{\bf e_1}}

\renewcommand{\b}{\beta}

\newcommand{\D}{\Delta}

\newcommand{\e}{\varepsilon}

\newcommand{\tht}{\theta}

\newcommand{\dd}{\,\text{\rm d}}             
\newcommand{\dB}{\dot B}

\newcommand{\cT }{{\cal T}}

\newcommand{\vphi}{\varphi}
\newcommand{\nn}{\nonumber}

\newcommand{\fC }{{\mathfrak C}}

\newcommand{\ssup}[1] {{\scriptscriptstyle{({#1}})}}

\makeatletter
\def\section{\@startsection{section}{1}{\z@}{-3.5ex plus -1ex minus 
 -.2ex}{2.3ex plus .2ex}{\bf}}
\def\subsection{\@startsection{subsection}{2}{\z@}{-3.25ex plus -1ex minus 
 -.2ex}{1.5ex plus .2ex}{\bf}}
\makeatother

\newcommand{\cvlaw}{\stackrel{\rm{ law}}{\longrightarrow}}
\newcommand{\eqlaw}{\stackrel{\rm{ law}}{=}}

\AtBeginDocument{
   \def\MR#1{}  }

\begin{document}

\pagestyle{myheadings}
\markboth{FC-CC-CM}{Fluctuation and Rate of Convergence for the SHE in Weak Disorder}

\title{Fluctuation and Rate of Convergence of the Stochastic Heat Equation in Weak Disorder}
\author{Francis Comets$^{1}$, Cl\'ement Cosco$^{1}$, Chiranjib Mukherjee$^{2}$}

\maketitle

{\footnotesize 
\noindent$^{~1}$Universit\'e Paris Diderot\\
Laboratoire de Probabilit\'es, Statistique et Mod\'elisation\\ LPSM (UMR 8001 CNRS, SU, UPD)\\
B\^atiment Sophie Germain, 8 place Aur\'elie Nemours, 75013 Paris\\
\noindent {\tt comets@lpsm.paris,  ccosco@lpsm.paris}
\\

\noindent$^{~2}$University of M\"unster\\
Fachbereich Mathematik und Informatik\\
Einsteinstra\ss e 62, M\"unster, D-48149\\
\noindent{\tt chiranjib.mukherjee@uni-muenster.de}

\begin{abstract} 
We consider the stochastic heat equation on $\R^d$ with multiplicative space-time white noise noise smoothed in space. 
For $d\geq 3$ and small noise intensity, the solution is known to converge to a strictly positive random variable as the smoothing parameter vanishes. In this regime, we study the rate of convergence and show that the pointwise 
fluctuations of the smoothened solutions as well as that of the underlying martingale of the Brownian directed polymer converge to a Gaussian limit. 
\end{abstract}
\\[-.3cm]

}


\section{Introduction and the result.}

We fix a spatial dimension $d\geq 3$ and consider a space-time Gaussian white noise $\dB$ on $\R_+\times \rd$. It is formally described by a family $\{\dB(\vphi)\}_{\varphi\in \mathcal S(\R_+\times \rd)}$
of Gaussian random variables  on a complete probability space $(\Omega,\mathcal B, \mathbb P)$ 
with mean 0 and covariance
\begin{equation} 
\begin{aligned}
\mathbb E\big[ \dB(\varphi_1)\,\, \dB(\varphi_2)\big]&=\E \left[ \int_0^\8 \int_{\rd} \dd t \,\dd x \,\, \dB(t,x)\,\vphi_1(t,x) \times \int_0^\8 \int_{\rd} \dd t \, \dd x \,\dB(t,x)\,  \varphi_2(t,x) \right] 
\\
&= \int_0^\8 \int_{\rd} \varphi_1(t,x) \varphi_2(t,x) \dd t \dd x \; .
\end{aligned}
\end{equation}
with $\varphi_1,\varphi_2$ being in the Schwartz space $\mathcal S(\R_+\times \rd)$ of all smooth and rapidly decreasing functions on $\mathbb R_+\times \mathbb R^d$. 
Throughout the article, $\mathbb E$ will denote expectation with respect to $\mathbb P$.
\medskip 

\noindent
We fix a  non-negative, smooth, spherically symmetric function $\phi: \R^d \to \R_+$ with support in the Euclidean ball $B(0,1/2)$ and normalized to have unit mass $\int_{\rd} \phi(x)\dd x=1$, and we define $\phi_\e(\cdot)=\e^{-d}\phi(\cdot/\e)$. Then,
\begin{equation} \nn
B_{t,\e}(x) = \dB \left( \varphi_{\e,t,x}\right) \;\qquad \mbox{with}\,\,\, \varphi_{\e,t,x} (s,y) =  {\un}_{[0,t]}(s) \phi_\e(y - x)\;, 
\end{equation}
denotes the spatially smoothened white noise, which is again a centered Gaussian process with covariance 
\begin{equation}\nn
\mathbb E[ B_{t,\e}(x) B_{s,\e}(y) ] = (s\wedge t) \big(\phi_\e\star \phi_\e\big)(x-y) = (s\wedge t) \, \e^{-d} V((x-y)/\e),
\end{equation}
where $V=\phi \star \phi$ is a smooth function supported in the ball $B(0,1)$. 
In particular, for any $x$, $(B_{t,\e}(x); t \geq 0)$ is a linear Brownian motion with diffusion constant $\e^{-d}V(0)$.
Following \cite{MSZ16}, we consider the  (smoothed) {\it multiplicative noise stochastic heat equation}: 

\begin{equation}\label{eq1-spde}
\dd u_{\e,t} = \frac12 \D u_{\e,t} \dd t + \b \e^{\frac{d-2}2}   u_{\e,t} \dd B_{\e,t} \;,\qquad\,\,   u_{\e,0}(x) =1\;,
\end{equation}
where $\beta > 0$ and  the stochastic differential is interpreted in the classical Ito sense. 
Then, by Feynman-Kac formula \cite[Theorem 6.2.5]{Kunita}
\begin{equation}\label{eq-FK}
u_{\e,t}(x)=E_x \bigg[ \exp\bigg\{\beta\e^\frac{d-2}{2} \,\int_0^t \int_{\rd} \, \phi_\e(W_{t-s}-y) \, \dB(s, y)\, \dd s\,\, \dd y -  \frac{\beta^2\,t\,\e^{-2}} 2\,\, V(0)\bigg\}\bigg]\;.
\end{equation} 
By time reversal, for any fixed $t>0$ and $\e>0$,
\begin{equation} \nn
 u_{\e,t} (\cdot)\,\, \eqlaw \,\,M_{\e^{-2}t}(\e^{-1} \cdot)\;,
\end{equation}
where 
\begin{equation} \label{eq:M}
M_T(x) = E_x\left[ \exp\left\{\b \int_0^T \int_{\rd} \phi(W_s-y) \dB(s,y) \dd s\, \dd y - \frac{\b^2T}{2}V(0)\right\} \right],
\end{equation}
and $E_x$ denotes the expectation for the law $P_x$ of a Brownian motion $W=(W_s)_{s\geq 0}$ starting from $x\in\rd$ and independent of the white noise $\dB$. See (Eq.(2.6) in \cite{MSZ16}) for details. Then it was shown \cite[Theorem 2.1 and Remark 2.2]{MSZ16} that for $\b>0$ sufficiently small and any test function $f\in\mathcal C_c^\infty(\rd)$, 
\begin{equation}\label{eq-MSZ}
\int_{\rd} u_\e(t,x) \, f(x)\,\dd x\to \int_{\rd} \overline u(t,x) \, f(x)\,\dd x
\end{equation}
 as $\e\to 0$ in probability, with $\overline u$ solving the heat equation 
\begin{equation}\label{eq-heat}
\partial_t \overline u= \frac 12 \Delta \overline u
\end{equation}
with unperturbed diffusion coefficient. Furthermore, it was also shown in \cite{MSZ16} that, with $\beta$ small enough, and 
for any $t>0$ and $x\in \rd$,  
$u_{\e,t}(x) $ converges in law to a non-degenerate random variable $M_\infty$ which is almost surely strictly positive, while $u_{\e,t}(x)$ converges in probability to zero if $\beta$ is chosen large. 
The law of $M_\infty$ was not determined in \cite{MSZ16}. 

Following the standard terminology used in the literature on discrete directed polymers, the Feynman-Kac representation \eqref{eq:M} relates $M_T$ (and thus, $u_{\e,t}$) to the {\it{(quenched) polymer partition function}}, and
existence of a strictly positive limit $M_\infty$ for small disorder strength $\beta$ is referred to as the {\it{weak-disorder regime}}, while for $\beta$ large, a vanishing partition function $\lim_{T\to\infty} M_T$ underlines the {\it{strong disorder phase}} (\cite{CSY04}). The polymer model corresponding to \eqref{eq:M} is known as Brownian directed polymer in a Gaussian environment, and the reader is refered to \cite{CC18} for a review of a similar model driven by a Poissonian noise.

\medskip 

Throughout this article we will focus deep inside the weak disorder regime, i.e., we will assume that $\beta$ is small enough and $M_\infty$ is a non-degenerate strictly positive random variable. The goal of the present article is to study the rescaled {\it{pointwise fluctuations}} 
$$
T^{\frac{d-2}4}\bigg(\frac{M_T(x)-M_\infty(x)}{M_T(x)}\bigg)
$$
for $x\in \rd$ as $T\to\infty$,  where $M_\8(x)$ is the a.s. limit of the positive martingale $M_T(x)$. Here is our first main result:
\medskip

\begin{theorem} \label{th:tclMT}
There exists $\beta_0\in (0,\infty)$ such that for $\beta<\beta_0$, 
and $x\in \rd$, with $d\geq 3$,
\begin{equation} \nn
T^{\frac{d-2}4}\bigg(\frac{M_T(x)-M_\infty(x)}{M_T(x)}\bigg)  \cvlaw N\big(0,\sigma^2(\beta)\big),
\end{equation}
as $T\to\infty$, 
where 
\begin{equation}\label{def-sigma-beta}
\sigma^2(\beta)=\frac{2}{(d-2)(2\pi)^{d/2}} \,\, \int_{\rd} \dd y \,\, V(\sqrt 2y)\,\,E_y\bigg[e^{\beta^2\int_0^\infty V(\sqrt 2 W_s)\,\dd s}\bigg].
\end{equation}
Moreover,  $T^{\frac{d-2}4}({M_T(x)-M_\infty(x))}/{M_\8(x)}$
 converges in law  to the same limit. 
\end{theorem}

We point that the variance $\sigma^2(\beta)$ in Theorem \ref{th:tclMT} becomes infinity for $\beta>\beta_c$ for some $\beta_c\in(0,\infty)$. We also conjecture that Theorem \ref{th:tclMT} holds in the whole ``$L^2$-region", i.e., for all $\beta$ such that $M_T$ remains bounded in $L^2$-norm, but such a proof seems out of reach at the present time.

We now derive the rate of convergence of the solution of the stochastic heat equation \eqref{eq1-spde}. For simplicity we focus on the case $x=0$, although the result easily extends to the general case. 
We will write $M_T=M_T(0)$. 
Note that 
\begin{equation}\label{def-F}
\mathscr F_T(\dB)= E_0\big[ \mathrm e^{\beta \dB ( \varphi)  - \b^2  V(0)T/2} \big] \;,\qquad {\rm with} \;\;\;  \varphi(t,x)= {\un}_{[0,T]}(t) \phi(W_t - x)\;,
\end{equation}
defines a measurable function $\mathscr F_T$ on the path space $(\Omega,\mathcal B,\mathbb P)$ of the white noise, such that $\mathscr F_T(\dB)=M_T$.
Since $M_T$ converges almost surely, we can select a representative $\mathscr F_\8(\dB)$ of 
the limit, i.e.,  $\mathscr F_\8(\dB)=M_\8$. 
Without loss of generality, we can assume that the definition of the white noise $\dB$ extends to negative times, and  for $\e>0, T>0$,  the random process 
$\dB^{\ssup{\e,T}}$ given by
  \begin{equation}\nn
\dB^{\ssup{\e,T}}(\varphi) = \e^{-\frac{d+2}{2} }\int_\R \int_{\rd} \varphi(T- \e^{-2}t, \e^{-1}y) \dB(t, y) \dd t\, \dd y
\end{equation}
is itself a Gaussian white noise. Moreover, in view of \eqref{def-F} and Brownian scaling (for $W$), the solution \eqref{eq-FK}
of the stochastic heat equation \eqref{eq1-spde} can be rewritten as
\begin{equation}\nn
u_{\e,t}(0) = \mathscr F_{\e^{-2}t} \big( \dB^{(\e,\e^{-2}t)} \big).
\end{equation}
Since $\dB^{\ssup{\e,T}} \eqlaw \dB$, we have for $T=\e^{-2}t$,
\begin{equation}\nn
\frac{\mathscr F_\8 ( \dB^{\ssup{\e,T}} )}{u_{\e,t}(0) }=
\frac{\mathscr F_\8 ( \dB^{\ssup{\e,T}} )}{\mathscr F_{\e^{-2}t} \left( \dB^{\ssup{\e,\e^{-2}t}} \right)}  \eqlaw
\frac{\mathscr F_\8 ( \dB)}{\mathscr F_{\e^{-2}t} ( \dB)}  = \frac{ M_\8}{M_{\e^{-2}t}}
\end{equation}
and the following result is a direct consequence of Theorem \ref{th:tclMT}.
\begin{cor}\label{coro:final}
Fix $\b < \b_0$ as in Theorem \ref{th:tclMT}. There exists a functional $\mathscr F_\8$ on the path space $(\Omega,\mathcal B, \mathbb P)$ which is positive and measurable with $\E [\mathscr F_\8(\dB)]=1$
and $M_\8 = \mathscr F_\8 (\dB)$, such that for all $t>0$, as $\e\to 0$,
\begin{equation}\nn
\e^{-\frac{d-2}{2} } \left( \frac{\mathscr F_\8 ( \dB^{\ssup{\e,\e^{-2}t}} )}{u_{\e,t}(0) } -1 \right)
 \cvlaw N\big(0,\sigma^2(\beta)t^{-\frac{d-2}{2}}\big) \;.
\end{equation}
Moreover, $\e^{-\frac{d-2}{2} } \left( \frac{u_{\e,t}(0) }{\mathscr F_\8 ( \dB^{\ssup{\e,\e^{-2}t}} )} -1 \right)$ converges in law to the same limit.
\end{cor}


While we do not discuss it in detail, Theorem \ref{th:tclMT} and Corollary \ref{coro:final} provide {\it{Edwards-Wilkinson}} type limit as $T\to\infty$ and $\e\to 0$ respectively. 
We mention two recent articles (\cite{GRZ17}, \cite{MU17}) where a similar problem has been studied 
in a different context. It was shown \cite[Theorem 1.2]{GRZ17} that, if $\beta>0$ is chosen sufficiently small, then for $f\in \mathcal C_c^\infty(\rd)$, 
\begin{equation}\label{eq1-GRZ}
\e^{1-\frac d2} \int_{\rd}\dd x \,\big[ u_\e(t,x)- \mathbb E(u_{\e}(t,x))\big] \, f(x) \Rightarrow \int_{\rd} \dd x\, \mathscr U(t,x) \,\, f(x)
\end{equation}
where $\mathscr U$ solves the heat equation with additive noise, or the Edwards-Wilkinson equation: 
\begin{equation}\label{eq2-GRZ}
\partial_t \mathscr U=\frac 12 \Delta \mathscr U+ \beta \sigma^2(\beta) \,\overline u \, \dB, \qquad \mathscr U(0,x)= 0,
\end{equation}
and $\overline u$ is the solution of the heat equation \eqref{eq-heat}. We remark that the nature of the results in \eqref{eq1-GRZ} and in Theorem \ref{th:tclMT},
as well as their proofs are different. In particular, in the present case we consider pointwise fluctuations of the form $M_T(x) - M_\infty(x)$ for $x\in \rd$ (i.e., we do not study the spatially smoothened averages of $M_T(x)-\mathbb E[M_T(x)]$).

The case when the noise $\dB$ is smoothened both in time and space has also recently been considered. If $F(t,x)= \int_{\rd}\int_0^\infty \phi_1(t-s) \phi_2(x-y) \dd B(s,y)$ is the mollified noise, and $\hat u_\e(t,x)=u(\e^{-2}t,\e^{-1}x)$ with $u$ solving $\partial_t u=\frac 12 \Delta u + \beta \, F(t,x) u$, then it was shown in \cite{M17} that for any $\beta>0$ and $x\in \rd$, $\frac 1 {\kappa(\e,t)}\,\, \mathbb E[\hat u_\e(t,x)] \to \hat u(t,x)$ as $\e\to 0$, where $\kappa(\e,t)$ is a divergent constant and $\hat u(t,x)$ solves the homogenized heat equation 
\begin{equation}\label{eq-heat-homog}
\partial_t \hat u= \frac 12 \mathrm{div}\big(\mathrm{a_\beta} \nabla \hat u\big)
\end{equation}
with diffusion coefficient $\mathrm a_\beta\in \mathbb R^{d\times d}$. It was then shown in \cite[Theorem 1.1]{GRZ17} 
that, for $\beta>0$ small enough, a result of the form \eqref{eq1-GRZ} holds also  for the rescaled and spatially averaged fluctuations $\e^{1-d/2}\int \dd x f(x)[\hat u_\e(t,x)-\mathbb E(\hat u_\e(t,x)]$,  
and the limit $\mathscr U$ again satisfies the additive noise stochastic heat equation $\partial_t \mathscr U=  \frac 12 \mathrm{div}\big(\mathrm{a_\beta} \nabla \mathscr U\big) + \beta \nu^2(\beta) \,\hat u \, \dB$ with diffusivity $\mathrm a_\beta$ and variance $\nu^2(\beta)$, and $\hat u$ solves \eqref{eq-heat-homog}. Note that, unlike \eqref{eq2-GRZ}, due to the presence of time correlations, in this case both the diffusion matrix and the variance of the noise are homogenized in the limit $\e\to 0$.

Finally we briefly comment on the strategy for the proof of Theorem \ref{th:tclMT} for which we loosely follow \cite{CL17} as a guiding philosophy.
The first step relies on a technical fact stated in Proposition \ref{prop:claim1} whose proof constitutes Section \ref{sec:prclaim1}.  
However, a key step for the proof of Theorem \ref{th:tclMT} is utterly disparate from \cite{CL17}. In particular, we do not take the approach via central limit theorem for martingales or use stable and mixing convergence (see \cite{HL15}) as in \cite{CL17} which can conceivably be adapted to the present case. Instead, we invoke techniques from stochastic calculus as in \cite{CN95} which are well-suited and efficient in the present scenario. The details can be found in Section \ref{sec-proof-theorem}.

 

\section{Proof of Theorem \ref{th:tclMT}.}\label{sec-proof-theorem}

\subsection{Rate of decorrelation.}
In this section we will provide the following elementary result, which provides an estimate on the asymptotic decorrelation of $u_\e(x)$ and $u_\e(y)$ as $\e\to 0$. This estimate also underlines the fact that smoothing $u_\e(x)$ w.r.t. 
any $f\in \mathcal C_c^\infty(\rd)$ makes $\int_{\rd} \dd x \, u_\e(x) \, f(x)$ deterministic (recall \eqref{eq-MSZ}). 
\begin{proposition}\label{prop:covar}  Let $d\geq 3$ and $\beta$ small enough.
\begin{itemize}
\item We have: 
\begin{equation}\label{eq:covM}
{\rm Cov}\big( M_\8(0), M_\8(x)\big)=
\begin{cases}
 E_{x/\sqrt{2}} \bigg[ \mathrm e^{\b^2 \int_0^\8 V(\sqrt 2 W_s) ds} -1 \bigg] \quad \forall x\in \rd, \\
\fC_1  \Big(\frac{1}{|x|}\Big)^{d-2}    \ \,\qquad\qquad\qquad\qquad\forall\, |x| \geq 1,
\end{cases}
\end{equation}
with $\fC_1= E_{\eu/\sqrt 2}\bigg[ \mathrm e^{\b^2 \int_0^\8 V(\sqrt 2 W_s) ds} -1 \bigg]$. 
\item Finally,
\begin{eqnarray} 
\big\|M_\8-M_T\big\|_2^2  \label{eq:asdifnorm}
&\sim& \fC_1 \fC_2 \E\big[M_\8^2\big] T^{-\frac{d-2}{2}}
 \qquad {\rm as} \; T \to \8.
\end{eqnarray}
with $\fC_2= E\big[ \big({\sqrt 2}/|Z|\big)^{d-2} \big]  $, where $Z$ is a centered Gaussian vector with covariance $I_d$.
\end{itemize}
\end{proposition}
\begin{proof}
For any Brownian path $W=(W_s)_{s\geq 0}$ we set 
\begin{equation} \label{eq-Phi}
\Phi_T(W)= \exp\left\{\b \int_0^T \int_{\rd} \phi(W_s-y) \dB(s,y) \dd s\, \dd y - \frac{\b^2T}{2 }V(0)\right\}
\end{equation}
and see that, for any $n\in \mathbb N$, 
\begin{equation} \label{eq:expgauss}
\E \bigg[   \prod_{i=1}^n \Phi_T( W^{\ssup i} ) \bigg] =
\exp \bigg\{ \b^2 \int_0^{T} \sum_{1 \leq i < j \leq n} V\big( W_s^{\ssup i}- W_s^{\ssup j} \big) ds  \bigg\}.
\end{equation}
By Markov property, for $0<S\leq \8$,
\begin{equation}\label{eq:markov}
M_{T+S}(x) = E_x\big[ \Phi_T(W) \,\,
M_S \circ \tht_{T,W_T}
\big]\;.
\end{equation}
where for any $t>0$ and $x\in \rd$, $\theta_{t,x}$ denotes the canonical spatio-temporal shift in the white noise environment. Then, ${\rm Var}\big( M_T(x)\big)= E_0\big[ \mathrm e^{\b^2 \int_0^T V(\sqrt 2 W_s) ds} -1 \big]$ and the first line of equation \eqref{eq:covM} follows from \eqref{eq:markov} and \eqref{eq:expgauss} with $n=2$ and Brownian scaling. 
Now, the second line of  \eqref{eq:covM}  follows by considering the hitting time of the unit ball for $\sqrt{2}W$ and spherical symmetry of $V$.

We now show \eqref{eq:asdifnorm} as follows. For two independent paths $W^{\ssup 1}$ and $W^{\ssup 2}$  (which are also independent of the noise $\dB$), we will denote by $\mathcal F_T$ the $\sigma$-algebra generated by 
both paths until time $T$. Then, by \eqref{eq:markov}, 
\begin{eqnarray*}
\|M_\8-M_T\|_2^2  &=& \E \bigg[E_0^{\otimes 2}\bigg\{ \Phi_T(W^{\ssup 1}) \Phi_T(W^{\ssup 2})  \left(M_\8 \circ \theta_{T,W_T^{\ssup 1}}-1\right)\left(M_\8  \circ \theta_{T,W_T^{\ssup 2}}-1\right)  \bigg\}\bigg]\\
 &=& E_0^{\otimes 2}\bigg[ \mathrm e^{\b^2 \int_0^{T} V( W_s^{\ssup 1}-W_s^{\ssup 2}) \dd s} \times {\rm Cov}\big(M_\8(W_T^{\ssup 1}),M_\8(W_T^{\ssup 2})\big)  \bigg]\\
&=& E_0^{\otimes 2}\bigg[ \mathrm e^{\b^2 \int_0^{T} V( W_s^{\ssup 1}-W_s^{\ssup 2}) \dd s} \times  E_0^{\otimes 2}\bigg[  {\rm Cov}\big(M_\8(W_T^{\ssup 1}),M_\8(W_T^{\ssup 2})\big)  \bigg \vert \mathcal F_T\bigg]\bigg] \\
&\sim& \fC_1  E_0^{\otimes 2}\left[ \mathrm e^{\b^2 \int_0^{T} V( W_s^{\ssup 1}-W_s^{\ssup 2}) \dd s} \times \left(\frac{ 2}{|W_T^{\ssup 1}-W_T^{\ssup 2}|}\right)^{d-2} \right]  \qquad\ ({\rm by\ }\eqref{eq:covM}) \\
&=& \fC_1  
 E_0\left[ \mathrm e^{\b^2 \int_0^T V(\sqrt 2 W_s) \dd s}  \bigg(\frac{\sqrt 2}{|W_T|}\bigg)^{d-2} \right] 
\end{eqnarray*}
Then \eqref{eq:asdifnorm} is proved once we show 
\begin{equation} \label{eq:asymptoticindependence1}
 E_0\bigg[ \mathrm e^{\b^2 \int_0^T V(\sqrt 2 W_s) \dd s}  \bigg(\frac{1}{|W_T|}\bigg)^{d-2} \bigg]
 \sim E_0\bigg[ \mathrm e^{\b^2 \int_0^{\8} V(\sqrt 2 W_s) \dd s} \bigg]  E_0\bigg[ \left(\frac{1}{|W_T|}\right)^{d-2} \bigg] 
\end{equation}
But as $T\to\infty$, 
\begin{equation} \nn
\left( \int_0^{T} V(\sqrt 2 W_s) \dd s, T^{-1/2}W_T \right)  \stackrel{\rm law}{\longrightarrow} \left( \int_0^{\8} V(\sqrt 2 W_s) \dd s, Z \right) 
\end{equation}
with $Z\sim N(0,I_d)$ being independent of the Brownian path $W$, and then \eqref{eq:asdifnorm} follows from the requisite uniform integrability 
\begin{equation}\label{eq-ui}
\sup_{T \geq 1}  E_0\left[\left(\mathrm e^{\b^2 \int_0^T V(\sqrt 2 W_s) \dd s}  \left(\frac{T^{1/2}}{|W_T|}\right)^{d-2} \right)^{1+\delta}\right] < \8 
\end{equation}
for $\delta>0$. By H\"older's inequality and Brownian scaling, for any $p,q\geq 1$ with $1/p+1/q=1$, 
$$
\begin{aligned}
\mbox{(l. h. s.) of \eqref{eq-ui}} &\leq E_0\bigg[\mathrm e^{q(1+\delta)\b^2 \int_0^T V(\sqrt 2 W_s) \dd s}\bigg]^{1/q}\,\, E_0\bigg[ \frac 1 {|W_1|^{p (1+\delta) (d-2)}} \bigg]^{1/p} \\
&\leq E_0\bigg[\mathrm e^{q(1+\delta)\b^2 \int_0^\infty V(\sqrt 2 W_s) \dd s}\bigg]^{1/q} \,\, \bigg[\int_{\mathbb R^d} \dd x \,\, \frac 1 {|x|^{p(1+\delta)(d-2)}} \mathrm e^{-|x|^2/2}\bigg]^{1/p} \\
&\leq C \bigg[\int_0^\infty \dd r  \,\, r^{d-1}\,\, \frac 1 {r^{p(1+\delta)(d-2)}} \mathrm e^{-r^2/2}\bigg]^{1/p} 
\end{aligned}
$$
Then the last integral is seen to be finite provided we choose $\delta>0$ and $p>1$ small enough so that $1< p(1+\delta) \leq \frac {d}{d-2}$.  
 \end{proof}

\subsection{Proof of Theorem \ref{th:tclMT}.}   

In this section we will prove Theorem \ref{th:tclMT}. We start by computing the stochastic differential and bracket of the martingale $M_T$ defined as follows: 
\begin{eqnarray} \nn
\dd M_T&=& \b E_0 \bigg[ \Phi_T(W)\,\, \int_{\rd} \phi(y-W_T) \dB(T, y) \dd T \, \dd y \bigg] \;,\\ \nn
\dd \langle  M\rangle_T &=& \b^2 E_0^{\otimes 2} \bigg[ \Phi_T(W^{\ssup 1})\Phi_T(W^{\ssup 2}) V\big(W_T^{\ssup 1}-W_T^{\ssup 2}\big)  \bigg]  \dd T \\ \label{eq:bracketM}
&=&  \b^2 M_T^{2} \times E_{0,\b,T}^{\otimes 2}  \left[ V(W_T^{\ssup 1}-W_T^{\ssup 2})  \right]  \dd T \;,
\end{eqnarray}
where $E_{0,\b,T}^{\otimes 2}$ is the expectation taken with respect to the product of two independent polymer measures, 
$$
P_{0,\beta,T}(\dd W^{\ssup i})=\frac 1 {Z_{\beta,T}} \,\, \exp\bigg\{ \beta \int_0^T \int_{\rd} \phi(W_s^{\ssup i}-y) \,\,\dB(y, s) \dd s \dd y \bigg\} \,\, P(\dd W^{\ssup i})\qquad i=1,2,
$$
with $ {Z_{\beta,T}}= e^{-\frac{\beta^2}{2}TV(0)}M_T $.
The proof of Theorem \ref{th:tclMT} splits into two main steps. The first step involves showing the following estimate whose proof consititues Section \ref{sec:prclaim1}: 
 \begin{proposition}\label{prop:claim1} There exists $\b_0\in(0,\infty)$, such that for all $\b < \b_0$, as $T \to \8$,
\begin{equation}\nn
T^{\frac{d}{2}} \left(\frac{\dd}{\dd t} \langle  M \rangle\right)_T -  \fC_3 M_T^2 \stackrel{L^2}{\longrightarrow} 0,  
\end{equation}
with  $\fC_3=\fC_3(\beta)=\frac {d-2}2 \, \sigma^2(\beta)$ and $\sigma^2(\beta)$ from \eqref{def-sigma-beta}.
\end{proposition}
For the second step, we define a sequence $\{G^{\ssup T}_\tau\}_{\tau\geq 1}$ of stochastic processes on time interval $[1, \8)$, with 
\begin{equation}\label{def:GT}
G^{\ssup T}_\tau=  T^{\frac{d-2}{4}}  \left( \frac{M_{\tau T}}{M_T} -1 \right) \;,\qquad \tau \geq 1.
\end{equation}
Then, for all $T$, $G^{\ssup T}$ is a continuous martingale for the filtration ${\mathcal B}^{(T)}=({\mathcal B}^{(T)}_\tau)_{\tau \geq 1}$, where ${\mathcal B}^{(T)}_\tau$ denotes
the $\sigma$-field generated by the white noise $\dB$ up to time $\tau T$. Then we need the following result, which 
provides convergence at the process level:
\begin{theorem} \label{th:cvGP}
For $\b < \beta_0$, as $T \to \8$, we have convergence
\begin{equation} 
G^{\ssup T} \stackrel{\rm law}{\longrightarrow} G
 \end{equation}
on the space of continuous functions on $[1, \8)$, where $G$ is a mean zero Gaussian process with independent increments  and variance
$$
g(\tau) = \sigma^2(\beta) \,\, [1-\tau^{-\frac{d-2}{2}}].
$$
\end{theorem}


\noindent{\bf{Proof of Theorem \ref{th:tclMT} (Assuming  Theorem \ref{th:cvGP}):}} 
We write 
\begin{eqnarray*}
T^{\frac{d-2}{4}}  \left( \frac{M_\8}{M_T} -1 \right)  &=& G^{\ssup T}(\8) \\&=& G^{\ssup T}(\tau) +\frac{ T^{\frac{d-2}{4}} [M_\8-M_{\tau T}] }{M_T} \;,
\end{eqnarray*}
and we consider the last term. By \eqref{eq:asdifnorm}, the numerator has $L^2$-norm tending to 0 as $\tau \to \8$ uniformly in $T \geq 1$ whereas the denominator has a positive limit.
Then, the last term vanishes in the double limit $T \to \8, \tau \to \8$, and therefore 
$$
\lim_{T \to \8} T^{\frac{d-2}{4}}  \left( \frac{M_\8}{M_T} -1 \right) = \lim_{\tau \to \8} \lim_{T \to \8} G^{\ssup T}(\tau) ,
$$
which is the Gaussian law with variance $g(\8)\!=\sigma^2(\beta)$ by Theorem \ref{th:cvGP}. Theorem \ref{th:tclMT} is proved. 
\qed
\noindent We now complete the 

\noindent{\bf{Proof of Theorem  \ref{th:cvGP} (Assuming Proposition \ref{prop:claim1}):}}  From the definition \eqref{def:GT} we compute the bracket 
of the square-integrable martingale $G^{\ssup T}$,
$$
\begin{aligned}
\langle G^{\ssup T}\rangle_\tau = \frac{T^\frac{d-2}{2}}{M_T^{2}}\,\, \langle M \rangle_{\tau T}&= \frac{T^\frac{d-2}{2}}{M_T^2} \int_1^{\tau T} 
\left(\frac{\dd}{\dd t} \langle  M \rangle\right)_{s} \dd s
\\  
&= \frac{T^\frac{d}{2}}{M_T^2} {\displaystyle  \int_1^\tau }
\left(\frac{\dd}{\dd t} \langle  M \rangle\right)_{\sigma T} \dd \sigma 
\end{aligned}
$$
by replacing the variables $s=\sigma T$. Then,
$$
\begin{aligned} 
\langle G^{\ssup T}\rangle_\tau -g(\tau)
& = {\displaystyle  \int_1^\tau }
\bigg[  \frac{(\sigma T)^\frac{d}{2}}{M_T^2}
\left(\frac{\dd}{\dd t} \langle  M \rangle\right)_{\sigma T} 
- \fC_3\bigg]
\sigma^{-d/2} \dd \sigma 
\\ 
&= {\displaystyle  \int_1^\tau }
\frac{M_{\sigma T}^2}{M_{ T}^2}
\left[  \frac{(\sigma T)^\frac{d}{2}}{M_{\sigma T}^2}
\left(\frac{\dd}{\dd t} \langle  M \rangle\right)_{\sigma T} 
- \fC_3\right]
\sigma^{-d/2} \dd \sigma 
+
 \frac{\fC_3}{M_T^2} 
{\displaystyle  \int_1^\tau }
\left[ M_{\sigma T}^{2}\!-\!{M_T^2}\right]
\sigma^{-d/2} d \sigma 
\\
& =: I_1 + I_2
\end{aligned}
$$
As $T \to \8$ the last integral vanishes in $L^2$ and $I_2$ vanishes in probability. 
For $\varepsilon \in (0,1]$, introduce the event 
$$
A_\varepsilon= \bigg\{ \sup\big\{M_t; t\in[0,\infty]\big\}\vee\sup\big\{M_t^{-1}; t\in[0,\infty]\big\} \leq \varepsilon^{-1}\bigg\}
$$
and observe that $\lim_{\varepsilon \to 0} \IP( A_\varepsilon) =1$ since $M_t$ is continuous, positive with a positive limit.  So, we can estimate the expectation of $I_1$ by
$$
\E \big[ {\bf 1}_{ A_\varepsilon}  \vert I_1 \vert
 \big]
\leq \frac{\tau}{\varepsilon^6} \,\, \sup_{ t \geq T}\bigg\{ \bigg \| 
t^{\frac{d}{2}} \left(\frac{\dd}{\dd t} \langle  M \rangle\right)_t \!-\!  \fC_3 M_t^2 
\bigg\|_1 \bigg\}  ,
$$
which vanishes by Proposition \ref{prop:claim1}. Thus, $\langle G^{\ssup T} \rangle \to g$ in probability. 
Since for the sequence of continuous martingales $G^{\ssup T}$  the brackets converge pointwise 
to a deterministic limit $g$, we derive that the sequence  $G^{\ssup T}$ itself converges in law to a Brownian motion with time-change given by $g$, that is, the process $G$ defined in the statement of Theorem \ref{th:cvGP} (see \cite[Theorem 3.11 in  Chapter 8]{JS87}), which is proved now.
\qed

\section{Proof of proposition \ref{prop:claim1}} \label{sec:prclaim1}

Denote for short by $L_T$ the quantity of interest,
\begin{eqnarray}
L_T &:=& T^{\frac{d}{2}} \left(\frac{\dd}{\dd t} \langle  M \rangle\right)_T -  \fC_3 M_T^2\nn \\ \nn
&=&  E_{0}^{\otimes 2}  \bigg[\Phi_T(W^{\ssup 1}) \Phi_T(W^{\ssup 2}) \bigg(T^{\frac{d}{2}}  V\big(W_T^{\ssup 1}\!-\! W_T^{\ssup 2}\big) - \fC_3 \bigg)\bigg],
\end{eqnarray}
and proceed in two steps.

\subsection{The first moment.}\label{sec:warmup}
We first want to show that

\begin{proposition}\label{cor-LT}
There exists $\b_1\in(0,\infty)$ such that for all $\b< \b_1$, if we choose
\begin{equation} \label{eq:fC_4}
\fC_3= (2\pi)^{-d/2} \int E_y \left[  \mathrm e^{\b^2\int_0^{\8} V(\sqrt 2  W_t)\dd t}\right]
V(\sqrt 2 y)  \dd y \;, 
\end{equation}
then $ \E(L_T)\to 0$ as $T\to\infty$.
\end{proposition}
\qed

The rest of Section \ref{sec:warmup} is devoted to the proof of Proposition \ref{cor-LT}. For any $t>s\geq 0$ and $x,y\in\rd$, we will denote by $P_{s,x}^{t,y}$ the law (and by 
$E_{s,x}^{t,y}$ the corresponding expectation) of the Brownian bridge starting at $x$ at time $s$ and conditioned to reach $y$ at time $t>s$. 
We will also write 
$$
\rho(t,x)= (2\pi t)^{-d/2} \mathrm e^{-|x|^2/2t}
$$
to be the standard Gaussian kernel. 

We note that 
\begin{eqnarray}\nn
\E (L_T) &=& E_{0}^{\otimes 2}  \bigg[ 
\mathrm e^{\b^2\int_0^T V(W^{\ssup 1}_t\!-\!W^{\ssup 2}_t)\dd t} \bigg( T^{\frac{d}{2}}  V\big(W_T^{\ssup 1}\!-\! W_T^{\ssup 2}\big) - \fC_3 \bigg)\bigg]
\\ \nn
&=&  E_{0}  \bigg[
\mathrm e^{\b^2\int_0^{T} V(\sqrt 2  W_t)\dd t} \bigg( T^{\frac{d}{2}}  V(\sqrt 2  W_T) - \fC_3 \bigg)\bigg]
\end{eqnarray}
and
\begin{equation}\nn
E_{0}  \bigg[  \mathrm e^{\b^2\int_0^{T} V(\sqrt 2  W_t)\dd t}T^{\frac{d}{2}}  V(\sqrt 2  W_T) \bigg] =
\int_{\rd}  V(\sqrt 2 y) E_{0,0}^{T,y}  \bigg[  \mathrm e^{\b^2\int_0^{T} V(\sqrt 2  W_t)\dd t} \bigg] T^{\frac{d}{2}}  \rho(T,y) \dd y.
\end{equation}
{Now, we fix a parameter $m=m(T)$, such that $m\to \infty$ and $m=o(T)$ as $T\to\infty$, which helps us prove Proposition \ref{cor-LT} in two steps:}
\begin{proposition}\label{prop-claim1}
For small enough $\b$, we have for any $y\in \rd$ and as $T\to\infty$,
\begin{equation} \nn
E_{0,0}^{T,y}   \bigg[ 
\mathrm e^{\b^2\int_0^{T} V(\sqrt 2  W_t)\dd t} \bigg] = \mathcal T_1 + o(1),
\end{equation}
where
$$\mathcal T_1=
E_{0,0}^{T,y}   \left[ 
\mathrm e^{\b^2\int_{[0,m] \cup [T\!-\!m,T]} V(\sqrt 2  W_t)\dd t} \right]. 
$$
\end{proposition}
\noindent 

\begin{proposition}\label{prop-claim2}
For small enough $\b$, we have as $T\to\infty$, 
\begin{eqnarray} \nn
\cT_1
&\sim &
E_{0,0}^{T,y}   \bigg[ 
\mathrm e^{\b^2\int_{[0,m] } V(\sqrt 2  W_t)\dd t} \bigg]
E_{0,0}^{T,y}   \bigg[ 
\mathrm e^{\b^2\int_{ [T\!-\!m,T]} V(\sqrt 2  W_t)\dd t} \bigg]\\ \nn
&\to&
E_{0}  \bigg[ \mathrm e^{\b^2\int_0^{\8} V(\sqrt 2  W_t)\dd t} \bigg]  E_{y}   \bigg[ 
\mathrm e^{\b^2\int_{0}^\8 V(\sqrt 2  W_t)\dd t}\bigg].
\end{eqnarray}
\end{proposition}

\medskip

We will provide some auxiliary results which will be needed to prove Proposition \ref{prop-claim1} and Proposition \ref{prop-claim2}.
First, we state a simple consequence of Girsanov's theorem:
\begin{lemma}\label{lemma1-claim1}
For any $s<t$ and $y,z\in \rd$, the Brownian bridge $P_{0,y}^{t,z}$ is absolutely continuous w.r.t. $P_{0,y}$ on the $\sigma$-field $\mathcal F_{[0,s]}$ generated by the Brownian path until time $s<t$, and 
\begin{equation} \label{lemma1-claim1-densityBBBM}
\begin{aligned}
\frac{ \dd P_{0,y}^{t,z}}{\dd P_{0,y}}\bigg|_{\mathcal F_{[0,s]}} &=\ \frac{\rho(t-s,z-W_s)}{\rho(t,z-y)}
&\leq \bigg(\frac t {t-s}\bigg)^{d/2} \exp\bigg\{\frac {|z-y|^2}{2t}\bigg\}.
\end{aligned}
\end{equation}
\end{lemma}
\qed
We will need the following version of Khas'minskii's lemma for the Brownian bridge: 
\begin{lemma}\label{lemma2-claim1}
If $E_0\bigg[2 \beta^2 \int_0^\infty V(\sqrt 2 W_s) \dd s\bigg] <1$, then 
\[
\sup_{z,x\in \mathbb R^d,t>0} E_{0,x}^{t,z}\bigg[\exp\bigg\{\beta^2 \int_0^t V(\sqrt 2 W_s) \dd s\bigg\}\bigg] <\infty.
\]
\end{lemma}
\begin{proof}
By Girsanov's theorem, for any $s<t$, $\alpha\in \R^d$ and $A\in \mathcal F_{[0,s]}$, 
\begin{equation}\label{eq1-lemma2-claim1}
P_{0,x}^{t,z}(A)= E^{(\alpha)}_x \bigg[ \frac{\rho^{(\alpha)}(t-s; z-W_s)}{\rho^{(\alpha)}(t,z-x)} \,\mathbf 1_A\bigg] 
\end{equation}
where $E^{(\alpha)}$ (resp. \unskip \  $P^{(\alpha)}$) refers to the expectation (resp. \unskip \ the probability) with respect to Brownian motion with drift $\alpha$ and transition density 
\[
\rho^{(\alpha)}(t,z)= \frac 1 {(2\pi t)^{d/2}} \exp\bigg\{- \frac{|z- t\alpha|^2}{2t}\bigg\}.
\]
With $\alpha=(z-x)/t$ and $s=t/2$, applying \eqref{eq1-lemma2-claim1}, we get 
\[
P_{0,x}^{t,z}(A) \leq 2^{d/2}\,\, P^{(\alpha)}_x (A).
\]
Replacing $A$ by $e^{2\beta^2 \int_0^{t/2} V(\sqrt 2 W_s) \dd s }$, we have
\[
\begin{aligned}
\sup_{z,x\in \mathbb R^d,t>0} E_{0,x}^{t,z}\bigg[\exp\bigg\{2\beta^2 \int_0^{t/2} V(\sqrt 2 W_s) \dd s\bigg\}\bigg] &\leq 2^{d/2} \sup_\alpha E^{(\alpha)}\bigg[\exp\bigg\{2\beta^2 \int_0^{t/2} V(\sqrt 2 W_s) \dd s\bigg\}\bigg] \\
&\leq 2^{d/2} \frac 1 {1-a} <\infty,
\end{aligned}
\]
where the second upper bound follows from Khas'minskii's lemma provided we have 
\[
2\beta^2 \sup_{x,\alpha} E^{(\alpha)}_x \bigg[ \int_0^\infty V(\sqrt 2 W_s) \dd s \bigg] \leq a <1.
\]
But since the expectation in the above display is equal to $\int_0^\infty \dd s \int_{\R^d} \dd z V(\sqrt 2 z)\,\, \rho^{(\alpha)} (s, z-x)$ and is maximal for $x=0$ and $\alpha=0$, the requisite condition reduces to 
\[
2\beta^2 E_0 \bigg[ \int_0^\infty V(\sqrt 2 W_s) \dd s \bigg]<1,
\]
which is satisfied by our assumption. Finally,  the lemma follows from the observation
 $$
 \exp\bigg\{\beta^2\int_0^t V(\sqrt 2 W_s) \,\dd s\bigg\}\leq \frac 12 \bigg[\exp\bigg\{2\beta^2\int_0^{t/2} V(\sqrt 2 W_s) \dd s \bigg\}+ \exp\bigg\{2\beta^2\int_{t/2}^{t} V(\sqrt 2 W_s) \dd s \bigg\}\bigg]
 $$
  combined with time reversibility of Brownian motion. 
\end{proof}
Recall that $V=\phi\star\phi$ is bounded and has support in a ball of radius $1$ around the origin, and therefore, for some constant $c, c^\prime>0$, and any $a>0$, 
$$
P_0\bigg[\int_m^\infty \dd s \, V(\sqrt 2 W_s) >a\bigg] \leq \frac {c}a \int_m^\infty \frac{\dd s}{s^{3/2}} \int_{B(0,1)} \dd y V(\sqrt 2 y) \exp\bigg\{-\frac{|y|^2}{2s}\bigg\} \leq \frac{c^\prime \|V\|_\infty} {am^{1/2}} \to 0 
$$
as $m\to\infty$, implying
\begin{lemma}\label{lemma2.5-claim1}
For any $a>0$, $\lim_{T\to\infty}\,\, P_0\big[\int_m^\infty \dd s \, V(\sqrt 2 W_s) >a\big] =0$.
\end{lemma}
By Lemma \ref{lemma2-claim1}, we also have
\begin{lemma}\label{lemma3-claim1}
For any $a>0$, 
$$
\lim_{T\to\infty}\sup_{z\in \mathbb R^d} P_{0,0}^{T,z}\bigg[\int_m^{T-m} V(\sqrt 2 W_s) \dd s >a \bigg] =0.
$$
\end{lemma}

\begin{proof}[{\bf{Proof of Proposition \ref{prop-claim1}}}]
Note that, for any $a>0$, we only need to show that

\[
\limsup_{T\to\infty} \sup_{{y\in\rd}}\,\, E_{0,0}^{T,y}  \bigg[ \mathrm e^{\b^2\int_0^{T} V(\sqrt 2  W_t)\dd t}\,\,\, \mathbf 1\bigg\{\int_m^{T-m} V(\sqrt 2 W_s)\,\dd s >a\bigg\}\bigg]  =0.
\]
But the above convergence follows by H\"older's inequality,  Lemma \ref{lemma2-claim1} and Lemma \ref{lemma3-claim1}.
\end{proof}

We now turn to the proof of 

\begin{proof}[{\bf{Proof of Proposition \ref{prop-claim2}}}]
Condition on the position of the Brownian bridge at time $T/2$, then use reversal property of the Brownian bridge and change of variable $z\to\sqrt{T} z$, to get:
\begin{align*}
&\cT_1 = \int_{\mathbb{R}^d} E_{0,0}^{T/2,z}\left[e^{\b^2\int_{[0,m] } V(\sqrt 2  W_t)\dd t} \right]E_{T/2,z}^{T,y}\left[\mathrm e^{\b^2\int_{[T-m,T] } V(\sqrt 2  W_t)\dd t} \right] \frac{\rho(T/2,z)\rho(T/2,y-z)}{\rho(T,y)}\dd z\\
& = \int_{\mathbb{R}^d} E_{0,0}^{T/2,z\sqrt{T}}\left[e^{\b^2\int_0^m V(\sqrt 2  W_t)\dd t} \right]E_{0,y}^{T/2,z\sqrt{T}}\left[\mathrm e^{\b^2\int_0^m V(\sqrt 2  W_t)\dd t} \right] \frac{\rho(1/2,z)\rho(1/2,z-y/\sqrt{T})}{\rho(1,y/\sqrt{T})}\dd z.
\end{align*}
We now claim that, for fixed $z$,
\begin{equation} \label{eq:BBToBM}
E_{0,y}^{T/2,z\sqrt{T}}\left[\mathrm e^{\b^2\int_0^m V(\sqrt 2  W_t)\dd t} \right] \sim E_{y}\left[\mathrm e^{\b^2\int_0^\infty V(\sqrt 2  W_t)\dd t} \right].
\end{equation}
Then, by dominated convergence theorem applied to the above integral, where the expectations in the integrand are bounded thanks to Lemma \ref{lemma2-claim1}, we obtain that:
\begin{align*}
\cT_1 \sim & \int_{\mathbb{R}^d} E_0\left[\mathrm e^{\b^2\int_0^\infty V(\sqrt 2  W_t)\dd t} \right]E_y\left[\mathrm e^{\b^2\int_0^\infty V(\sqrt 2  W_t)\dd t} \right] \frac{\rho(1/2,z)\rho(1/2,z)}{\rho(1,0)}\dd z\\
& = E_0\left[\mathrm e^{\b^2\int_0^\infty V(\sqrt 2  W_t)\dd t} \right]E_y\left[\mathrm e^{\b^2\int_0^\infty V(\sqrt 2  W_t)\dd t} \right].
\end{align*}
To prove \eqref{eq:BBToBM}, we use Lemma \ref{lemma1-claim1}:
\begin{align*}
&E_{0,y}^{T/2,z\sqrt{T}}\left[\mathrm e^{\b^2\int_0^m V(\sqrt 2  W_t)\dd t} \right] \\
& = \frac{1}{\rho(T/2,z\sqrt{T}-y)} E_{y} \left[\mathrm e^{\b^2\int_0^m V(\sqrt 2  W_t)\dd t} \rho(T/2-m,z\sqrt{T}-\sqrt{2}W_m) \right]\\
& = \frac{1}{\rho(1/2,z-y/\sqrt{T})\left(\pi(1-\frac{2m}{{T}})\right)^{d/2}} E_{y} \left[\mathrm e^{\b^2\int_0^m V(\sqrt 2  W_t)\dd t} \mathrm e^{-\frac{|z-\sqrt{2/T}\,W_m|^2}{1-2m/T}}\right].
\end{align*}
By monotone convergence and the fact that $m=o(T)$, we obtain:
\[P\text{-a.s.}\quad 
\mathrm e^{\b^2\int_0^m V(\sqrt 2  W_t)\dd t} \to \mathrm e^{\b^2\int_0^\infty V(\sqrt 2  W_t)\dd t}\quad \text{and} \quad
\mathrm e^{-\frac{|z-\sqrt{2/T}\,W_m|^2}{1-2m/T}} \to \mathrm e^{-2z^2}.
\]
Then, we have the following uniform integrability property for small $\delta >0$ and small $\b$:
\[
E_y\left[\left(\mathrm e^{\b^2\int_0^m V(\sqrt 2  W_t)\dd t} \mathrm e^{-\frac{|z-\sqrt{2/T}\,W_m|^2}{1-2m/T}}\right)^{1+\delta}\right] \leq E_y\left[\mathrm e^{(1+\delta)\b^2\int_0^\infty V(\sqrt 2  W_t)\dd t}\right]<\infty.
\]
 Hence,
\begin{equation*}
E_{0,y}^{T/2,z\sqrt{T}}\left[\mathrm e^{\b^2\int_0^m V(\sqrt 2  W_t)\dd t} \right] \to \frac{\mathrm e^{-2z^2}}{\rho(1/2,z)\pi^{d/2}} E_{y} \left[\mathrm e^{\b^2\int_0^\infty V(\sqrt 2  W_t)\dd t} \right] = E_{y} \left[\mathrm e^{\b^2\int_0^\infty V(\sqrt 2  W_t)\dd t} \right].
\end{equation*}

\end{proof}

\subsection{Second moment.}\label{sec:second}

The goal of this section is to show 

\begin{proposition}\label{prop-second-moment}
There exists $\b_0\in (0,\infty)$, such that for all $\b< \b_0$, $\E(L_T^2) \to 0$. 
\end{proposition}

For the proof of the above result, it is enough to show that $\limsup_{T\to\infty} \mathbb E(L_T^2)\leq 0$. Computing second moment, we get an integral over four independent Brownian paths: 
\begin{eqnarray}\label{eq1-second}
\E(L_T^2) &= E_{0}^{\otimes 4}   \bigg[ \prod_{i\in\{1,3\}} \bigg( T^{\frac{d}{2}}  V(W^{\ssup i}_T\!-\!W^{\ssup{i+1}}_T) - \fC_3 \bigg)  
\mathrm e^{\b^2 \sum_{1\leq i < j \leq 4} \int_0^T V(W^{\ssup i}_t\!-\!W^{\ssup j}_t)\dd t} \bigg] 
\nn \\
&= E_{0}^{\otimes 4} \bigg[  \; \prod_{i\in\{1,3\}} \bigg\{\mathrm e^{\b^2\int_0^T V(W^{\ssup i}_t\!-\!W^{\ssup{i+1}}_t)\dd t}  \bigg(T^{\frac{d}{2}}  V(W^{\ssup i}_T\!-\!W^{\ssup{i+1}}_T) - \fC_3\bigg) \bigg\}
\nn \\
&\qquad\times \mathrm e^{\b^2\sum^{*} \int_0^T V(W^{\ssup i}_t\!-\!W^{\ssup j}_t)\dd t}\; \bigg]
\nn
\end{eqnarray}
where the sum $\sum^{*} $ is considered for $4$ pairs $(i,j), {1\leq i < j \leq 4}$ different from $(1,2)$ and $(3,4)$.
\medskip

Throughout the rest of the article, for notational convenience, we will write 
\begin{equation}\label{eq-H-m}
\begin{aligned}
&H_m= \mathrm e^{\b^2 \sum_{1\leq i < j \leq 4}  \int_0^m V(W^{\ssup i}_t\!-\!W^{\ssup j}_t)\dd t} \, , \\
&H_{T-m,T}=\prod_{i\in\{1,3\}}\bigg\{\mathrm e^{\b^2  \int_{T-m}^T V(W^{(i)}_t\!-\!W^{(i+1)}_t)\dd t}\,\,\bigg(T^{d/2} V\left(W^{(i)}_T\!-\!W^{(i+1)}_T\right) - \fC_4 \bigg)\bigg\}.
\end{aligned}
\end{equation}
We will now estimate each term in the expectation in \eqref{eq1-second}. 
Proposition \ref{prop-claim3} stated below enables us to neglect the contributions
of $\int_m^{T-m} V(W^{\ssup i}_t\!-\!W^{\ssup j}_t)\dd t $ for all $i,j$ 
and of
$\int_{T-m}^{T} V(W^{\ssup i}_t\!-\!W^{\ssup j}_t)\dd t $ for all $(i,j) \neq (1,2), (3,4)$. 
More precisely, we want to show that 
\begin{proposition}\label{prop-claim3}
For $m={m(T)}$ as above, there exists a constant $C>0$ such that, for small enough $\b$,
\[
\limsup_{T\to\infty} \E L_T^2 \leq C \limsup_{T\to\infty}\cT_2,  
\]
where
\begin{equation}\label{eq1-T2}
\begin{aligned}
 \mathcal T_2&=  E_{0}^{\otimes 4}   \big[  H_m \,\, H_{T-m,T}\big].
 \end{aligned}
\end{equation}
\end{proposition}
Then, Proposition \ref{prop-second-moment} will be a consequence of 
\begin{proposition}\label{prop-claim4}
For small enough $\b$, we have as $T\to\infty$:
\begin{equation} \label{eq:prop-claim4}
\begin{aligned}
 \cT_2 &=  E_{0}^{\otimes 4}   \bigg[ \mathrm e^{\b^2 \sum_{1\leq i < j \leq 4}  \int_0^\infty V(W^{(i)}_t\!-\!W^{(j)}_t)\dd t} \bigg] \\
 &\qquad\qquad\times \left[E_{0}^{\otimes 2}   \left(  \mathrm e^{\b^2\int_{T-m}^T V(W^{(1)}_t\!-\!W^{(2)}_t)\dd t} \left[ T^{\frac{d}{2}}  V(W^{(1)}_T\!-\!W^{(2)}_T) - \fC_3 \right] \right) \right]^2 + o(1).
 \end{aligned}
\end{equation}
As a result, $\cT_2 \to 0$.
 \end{proposition}

\subsection{Proof of Proposition \ref{prop-claim3}.}\label{sec:proof-claim3}


We pick up from the first display in \eqref{eq1-second}, and write 
\begin{equation}\label{eq2-second}
\E(L_T^2) = \mathcal T_2^{\mathrm{(I)}} + \mathcal T_2^{\mathrm{(II)}}
\end{equation}
where 
\begin{equation}\label{eq1-T2-II}
\begin{aligned}
\mathcal T_2^{\mathrm{(II)}} &=  E_{0}^{\otimes 4} \bigg[ \mathrm e^{\b^2 \sum_{1\leq i < j \leq 4}  \int_0^T V(W^{\ssup i}_t\!-\!W^{\ssup j}_t)\dd t} \,\,
 \prod_{i=\{1,3\}} \bigg(T^{d/2} V(W^{\ssup i}_T\!-\!W^{\ssup{i+1}}_T) - \fC_3 \bigg) \\
&\qquad\qquad \times  \mathbf 1\bigg\{\int_{m}^{T-m} V(W^{\ssup i}_t\!-\!W^{\ssup j}_t)\dd t \geq a\,\,\mbox{for some}\,\, 1\leq i < j \leq 4\bigg\}\bigg] 
\end{aligned}
\end{equation}
and $\mathcal T_2^{\mathrm{(I)}}$ is defined canonically. We claim that, 
\begin{equation}\label{claim-T2-II}
\limsup_{T\to\infty} \, \mathcal T_2^{\mathrm{(II)}}=0.
\end{equation}
To prove the above claim, in \eqref{eq1-T2-II} we first estimate, using that $V,\fC_3\geq 0$, 
$$
\prod_{i\in\{1,3\}} \bigg[T^{d/2} V(W^{\ssup i}_T\!-\!W^{\ssup{i+1}}_T) - \fC_3\bigg] \leq  \prod_{i\in\{1,3\}} \bigg[T^{d/2} V(W^{\ssup i}_T\!-\!W^{\ssup{i+1}}_T)\bigg] + \fC_3^2.
$$
Note that,  
$$
\begin{aligned}
\fC_3^2 \limsup_{T\to\infty}\,\, E_{0}^{\otimes 4}& \bigg[ \mathrm e^{\b^2 \sum_{1\leq i < j \leq 4}  \int_0^T V(W^{\ssup i}_t\!-\!W^{\ssup j}_t)\dd t} \\
&\qquad \times \mathbf 1\bigg\{\int_{m}^{T-m} V(W^{\ssup i}_t\!-\!W^{\ssup j}_t)\dd t \geq a\,\,\mbox{for some}\,\, 1\leq i < j \leq 4\bigg\}\bigg] =0,
\end{aligned}
$$
by H\"older's inequality combined with Khas'minskii's lemma and Lemma \ref{lemma2.5-claim1}. Therefore, 
\begin{equation}\label{eq2-T2-II}
\begin{aligned}
\mathcal T_2^{\mathrm{(II)}} &\leq  E_{0}^{\otimes 4} \bigg[ e^{\b^2 \sum_{1\leq i < j \leq 4}  \int_0^T V(W^{(i)}_t\!-\!W^{(j)}_t)\dd t} \,\,
 \prod_{i=\{1,3\}} \bigg(T^{d/2} V(W^{(i)}_T\!-\!W^{(i+1)}_T)  \bigg) \\
&\qquad\qquad \times  \mathbf 1\bigg\{\int_{m}^{T-m} V(W^{(i)}_t\!-\!W^{(j)}_t)\dd t \geq a\,\,\mbox{for some}\,\, 1\leq i < j \leq 4\bigg\}\bigg] + o(1).
\end{aligned}
\end{equation}
Next, we switch from free Brownian motion to the Brownian bridge in the first term on the right hand side above, such that, writing $\mathbf y=(y_1,\dots,y_4)$, we get
$$
\begin{aligned}
\mathcal T_2^{\mathrm{(II)}}
&\leq  \int_{(\mathbb R^d)^4} \dd \mathbf y\,\, \prod_{i\in\{1,3\}}\bigg(T^{-d/2} \mathrm e^{-\frac{|y_i|^2+|y_{i+1}|^2}{2T}}\, V(y_i-y_{i+1})\bigg) \\
& \qquad\bigotimes_{i=1}^4 E_{0,0}^{T,y_i} \bigg[ \mathrm e^{\b^2 \sum_{1\leq i < j \leq 4}  \int_0^T V(W^{(i)}_t\!-\!W^{(j)}_t)\dd t}  \\
 &\qquad\qquad\qquad\times\mathbf 1\bigg\{\int_{m}^{T-m} V(W^{(i)}_t\!-\!W^{(j)}_t)\dd t \geq a\,\,\mbox{for some}\,\, 1\leq i < j \leq 4\bigg\}\bigg] +o(1).
\\
\end{aligned}
$$
Note that $V$ has support in a ball of radius $1$ around $0$. We now again use H\"older's inequality which, combined with Lemma \ref{lemma2-claim1} and Lemma \ref{lemma3-claim1} finish the proof of \eqref{claim-T2-II}.

We now turn to estimate $\mathcal T_2^{\mathrm{(I)}}$, which, by the second display in \eqref{eq1-second}, \eqref{eq2-second} and \eqref{eq1-T2-II} is given by 
\[
\begin{aligned}
\mathcal T_2^{\mathrm{(I)}} &=\,\,  E_{0}^{\otimes 4} \bigg[ H_m\,\, H_{T-m,T}\,\,
\,\, \bigg\{\mathrm e^{\b^2 \sum^\star \int_{T-m}^T V(W^{(i)}_t\!-\!W^{(j)}_t)\dd t}\bigg\} \\
&\qquad\qquad \times \mathrm e^{\b^2 \sum_{1\leq i < j \leq 4}  \int_m^{T-m} V(W^{(i)}_t\!-\!W^{(j)}_t)\dd t} \mathbf 1\bigg\{\int_{m}^{T-m} V\left(W^{(i)}_t\!-\!W^{(j)}_t\right)\dd t \leq a\,\,\mbox{for all}\,\, 1\leq i < j \leq 4\bigg\}\bigg] 
\\
&\leq \mathrm e^{6\b^2 a} E_{0}^{\otimes 4} \bigg[ H_m\,\, H_{T-m,T}\,\,
\,\, \bigg\{\mathrm e^{\b^2 \sum^\star \int_{T-m}^T V(W^{(i)}_t\!-\!W^{(j)}_t)\dd t}\bigg\}\bigg].
\end{aligned}
\]
We again want to ignore the contribution of the last term. But this can be done exactly as in the way we estimated $\mathcal T_2^{\mathrm{(II)}}$ by splitting interactions for $(i,j)\in \sum^\star$ when 
$ \int_{T-m}^T V(W^{\ssup i}_t\!-\!W^{\ssup j}_t)\dd t \geq a$ and $\int_{T-m}^T V(W^{\ssup i}_t\!-\!W^{\ssup j}_t)\dd t \leq a$. In order to avoid repetition we omit the details, and conclude the proof of Proposition \ref{prop-claim3}.
\qed

\medskip

\subsection{Proof of Proposition \ref{prop-claim4}.}\label{sec:proof-claim4}

If we denote by $\mathcal F_{[0,T/2]}$ the $\sigma$-algebra generated by all four Brownian paths until time $T/2$, then, using Markov's property, 
\[
\begin{aligned}
\mathcal T_2= E_0^{\otimes 4}[H_m\,\, H_{T-m,T}] 
&=E_0^{\otimes 4}\left[E_0^{\otimes 4}\left(H_m\,\, H_{T-m,T}\middle| \big(W_{T/2}^{\ssup i}\big)_{i=1}^4\right)\right] \\
&=E_0^{\otimes 4}\left[E_0^{\otimes 4}\left\{E_0^{\otimes 4}\left(H_m\,\, H_{T-m,T}\middle| \mathcal F_{[0,T/2]}\right) \middle | \big(W_{T/2}^{(\ssup i}\big)_{i=1}^4\right\} \right] \\
&=E_0^{\otimes 4}\left[E_0^{\otimes 4}\left\{H_m\middle| \big(W_{T/2}^{\ssup i}\big)_{i=1}^4\right\} E_0^{\otimes 4}\left\{H_{T-m,T}\middle| \big(W_{T/2}^{\ssup i}\big)_{i=1}^4\right\}\right].
\end{aligned}
\]
We will prove that there exists a constant $C<\infty$, such that: 
\begin{equation*}
\begin{gathered}
\text{(i)} \sup_{T>0} E_0^{\otimes 4}\left\{H_m\middle| \big(W_{T/2}^{\ssup i}\big)_{i=1}^4\right\} \leq C \qquad \text{(ii)} \sup_{T>0} E_0^{\otimes 4}\left\{H_{T-m,T}\middle| \big(W_{T/2}^{\ssup i}\big)_{i=1}^4\right\} \leq C,\\
\text{(iii)} \  E_0^{\otimes 4}\left\{H_m\middle| \big(W_{T/2}^{\ssup i}\big)_{i=1}^4\right\} \cvlaw E_0^{\otimes 4}\left[H_\infty\right], \text{ as } T\to\infty.
\end{gathered}
\end{equation*}
where $H_\infty$ is defined as $H_m$ with the time interval $[0,m]$ replaced by $[0,\infty)$, recall \eqref{eq-H-m}. 

Let us first conclude the proof of Proposition \ref{prop-claim4} assuming the above three assertions. The difference of the two first terms in \eqref{eq:prop-claim4} writes:
\begin{align*}
&\mathcal T_2 - E_0^{\otimes 4}\left[ H_\infty \right] E_0^{\otimes 4}\left[H_{T-m,T}\right]
 \\
&= E_0^{\otimes 4}\left[\left(E_0^{\otimes 4}\left\{H_m\middle| \big(W_{T/2}^{\ssup i}\big)_{i=1}^4\right\} - E_0^{\otimes 4}\left[ H_\infty \right]\right)
 E_0^{\otimes 4}\left\{H_{T-m,T}\middle| \big(W_{T/2}^{\ssup i}\big)_{i=1}^4\right\}\right],
\end{align*}
which goes to $0$ as $T\to\infty$ by (i)-(iii), proving \eqref{eq:prop-claim4}. Finally, computations of Section 3.1 ensure that:
\[
\bigg[E_0^{\otimes 2}\bigg\{\mathrm e^{\beta^2 \int_{T-m}^T V(W^{\ssup 1}_s- W^{\ssup 2}_s) \dd s}\,\,\bigg(T^{d/2}V\left(W^{\ssup 1}_T- W^{\ssup 2}_T\right)- \fC_3\bigg)\bigg\}\bigg] \underset{T\to\infty}{\longrightarrow}0.
\]

We now owe the reader the proofs of (i)-(iii). To prove (i), we use H\"older's inequality to get 
\[
\begin{aligned}
E_0^{\otimes 4}\left\{H_m\middle| \big(W_{T/2}^{(i)}\big)_{i=1}^4\right\} &\leq \prod_{1\leq i < j \leq 4} E_0^{\otimes 4}\left[\mathrm e^{6\beta^2 \int_0^m V(W^{\ssup i}_t\!-\!W^{\ssup j}_t)\dd t}\middle| \big(W_{T/2}^{\ssup i}\big)_{i=1}^4\right]^{1/6} \\
&= \prod_{1\leq i < j \leq 4} E_{0,0}^{T/2,W_{T/2}^{\ssup i}-W_{T/2}^{\ssup j}} \bigg[\mathrm e^{6\beta^2 \int_0^m V(\sqrt 2 W_t)\dd t}\bigg]^{1/6} \\
&\leq \sup_{T,z} E_{0,0}^{T/2,z} \bigg[e^{6\beta^2 \int_0^{T/2} V(\sqrt 2 W_t)\dd t}\bigg] <\infty,
\end{aligned}
\]
by Lemma \ref{lemma2-claim1}. For (ii), we note that by Markov's property, 
\[
\begin{aligned}
E_0^{\otimes 4}\left\{H_{T-m,T}\middle| \big(W_{T/2}^{\ssup i}\big)_{i=1}^4\right\} &=\prod_{i\in\{1,3\}} E_{W_{T/2}^{\ssup i}-W_{T/2}^{\ssup{i+1}}} \left[ \mathrm e^{\beta^2 \int_{T/2-m}^{T/2} V(\sqrt 2 W_t) \dd t}\,\left(T^{d/2} V\left(\sqrt 2 W_{T/2}\right)- \fC_3\right)\right].
\end{aligned}
\]
We have:
\[
\fC_3\, E_{W_{T/2}^{\ssup i}-W_{T/2}^{\ssup{i+1}}} \bigg[ \mathrm e^{\beta^2 \int_{T/2-m}^{T/2} V(\sqrt 2 W_t) \dd t}\bigg]\leq \fC_3 \sup_{z} E_z \bigg[\mathrm e^{\beta^2 \int_0^\infty V(\sqrt 2 W_t) \dd t}\bigg]<\infty,
\]
while, for some constant $C'>0$, 
\[
\begin{aligned}
&E_{W_{T/2}^{\ssup i}-W_{T/2}^{\ssup{i+1}}} \bigg[ e^{\beta^2 \int_{T/2-m}^{T/2} V(\sqrt 2 W_t) \dd t}\,\, T^{d/2}V\left(\sqrt 2 W_{T/2}\right)\bigg] \\
&\leq C'\int_{\mathbb R^d} \dd z \,\,\, E_{0,W_{T/2}^{\ssup i}-W_{T/2}^{\ssup{i+1}}}^{T/2,z}\,\,\, \bigg[ \mathrm e^{\beta^2 \int_{T/2-m}^{T/2} V(\sqrt 2 W_t) \dd t}\bigg]\,\, V\left(\sqrt 2 z\right) \\
&\leq C'  \sup_{T,y,z}\,\,  E_{0,y}^{T/2,z} \bigg[ \mathrm e^{\beta^2 \int_0^{T/2} V(\sqrt 2 W_t)\,\,\dd t}\bigg] \,\, \int_{B(0,1)} \dd z \,\, V\left(\sqrt 2 z\right)\\
&<\infty,
\end{aligned}
\]
again by Lemma \ref{lemma2-claim1}. 

Finally, to prove (iii), we fix any smooth test function $f:\mathbb R\to\mathbb R$, so that 
\begin{align}
E_0^{\otimes 4} \bigg[ f\left(E_0^{\otimes 4}\left\{H_m\middle| \big(W_{T/2}^{\ssup i}\big)_{i=1}^4\right\}\right)\bigg]&= \int_{(\mathbb R^d)^4} \dd \mathbf y \,\, f\bigg(E_{0,0}^{T/2,\mathbf y}\left[H_m\right]\bigg) \prod_{i=1}^4 \rho(T/2,y_i) \nn\\
&= \int_{(\mathbb R^d)^4} \dd \mathbf z\,\,  f\bigg(E_{0,0}^{T/2,\sqrt T \mathbf z}\left[H_m\right]\bigg)  \prod_{i=1}^4 \rho(1/2,z_i). \label{eq-claim3-prop-claim4}
\end{align}
Now, letting $T\to\infty$, we get similarly to \eqref{eq:BBToBM} that $E_{0,0}^{T/2,\sqrt T \mathbf z}\left[H_m\right]\to E_0^{\otimes 4}\left[H_\infty\right]$. By dominated convergence, the RHS of \eqref{eq-claim3-prop-claim4} converges to $f\left(E_0^{\otimes 4}\left[H_\infty\right]\right)$, implying (iii).
\qed

\noindent{\bf{Acknowledgement:}} The authors would like to thank the ICTS, Bangalore for the hospitality during the program {\it{Large deviation theory in statistical physics}}
(ICTS/Prog-ldt/2017/8), 
where the present work was initiated.

\end{document}